\documentclass[12pt,a4paper]{article}
\usepackage{indentfirst}
\usepackage{amsmath,amssymb,amsfonts,amsthm,graphics}
\usepackage{makeidx}
\usepackage{ifpdf}
\newtheorem{thm}{Theorem}

\newtheorem{lem}{Lemma}

\theoremstyle{definition}

\newtheorem{example}{Example}

\def\-{\mbox{--}}

\newtheorem{pro}{Proposition}

\newtheorem{cor}{Corollary}

\newtheorem{obs}{Observation}

\begin{document}
\title{\Large\bf The generalized 3-edge-connectivity\\ of lexicographic product graphs\footnote{Supported by NSFC No.11371205.}}
\author{\small  Xueliang~Li, Jun Yue, Yan~Zhao\\
\small Center for Combinatorics and LPMC-TJKLC\\
\small Nankai University, Tianjin 300071, China\\
\small lxl@nankai.edu.cn; yuejun06@126.com; zhaoyan2010@mail.nankai.edu.cn}
\date{}
\maketitle
\begin{abstract}
The generalized $k$-edge-connectivity $\lambda_k(G)$ of a graph $G$
is a generalization of the concept of edge-connectivity. The
lexicographic product of two graphs $G$ and $H$, denoted by $G\circ
H$, is an important graph product. In this paper, we mainly study
the generalized 3-edge-connectivity of $G \circ H$, and get upper
and lower bounds of $\lambda_3(G \circ H)$. Moreover, all bounds are
sharp.



\end{abstract}

\section{Introduction}
All graphs considered in this paper are simple, finite and
undirected. We follow the terminology and notation of Bondy and
Murty \cite{Bondy}. For a graph $G$, the local edge-connectivity
between distinct vertices $u$ and $v$, denoted by $\lambda(u,v)$, is
the maximum number of pairwise edge disjoint $uv$-paths. A
nontrivial graph $G$ is $k$-$edge$-$connected$ if $\lambda(u,v)\geq
k$ for any two distinct vertices $u$ and $v$ of $G$. The edge-connectivity $\lambda(G)$ of a graph $G$ is the maximum value of $k$
for which $G$ is $k$-edge-connected, see \cite{Oellermann1}.

Naturally, the concept of edge-connectivity can be extended to a new
concept, the generalized $k$-edge-connectivity, which was
introduced by Li et al. \cite{LLMS}. For a graph $G=(V,E)$ and a set $S\subseteq V(G)$
of at least two vertices, \emph{an $S$-Steiner tree} or \emph{a
Steiner tree connecting $S$} (or simply, \emph{an $S$-tree}) is a
such subgraph $T=(V',E')$ of $G$ that is a tree with $S\subseteq
V'$. Let $\lambda(S)$ denote the
maximum number of pairwise edge-disjoint Steiner trees $T_1, T_2,
\cdots,$ $ T_{\ell}$ in $G$ such that if $E(T_i)\cap
E(T_j)=\emptyset$ and $S\subseteq V(T_i)\cap V(T_j)$ for any pair of
distinct integers $i$, $j$ with $1\leq i,j\leq \ell$. Then the \emph{generalized
$k$-edge-connectivity} $\lambda_k(G)$ of $G$ is defined as
$\lambda_k(G)=\min\{\lambda(S)\,|\, S\subseteq V(G), |S|=k\}$.
Obviously, $\lambda_2(G)=\lambda(G)$. Set $\lambda_k(G)=0$ if $G$ is
disconnected. Some results of the edge-connectivity and the
generalized edge-connectivity can refer to
\cite{LLMS,Oellermann1,Volkmann} for details.

The generalized edge-connectivity is closely linked to an important
problem \emph{Steiner tree packing problem}, which asks for finding
a set of maximum number of edge-disjoint $S$-trees in a given graph
$G$ where $S\subseteq V(G)$, see \cite{Kriesell1,West,Grotschel1}.
An extreme of Steiner tree packing problem is called the
\emph{Spanning tree packing problem} where $S=V(G)$. For any graph
$G$, the \emph{spanning tree packing number} or \emph{$STP$ number},
is the maximum number of edge-disjoint spanning trees contained in
$G$. For the $STP$ number, we refer to \cite{OY,
Palmer,Barden,Fragopoulou,Itai,Roskind, Wang}.
The difference between the Steiner tree packing problem and the
generalized edge-connectivity is as follows: the former problem
studies local properties of graphs since $S$ is given beforehand,
while the latter problem focuses on global properties of graphs
since it first needs to compute the maximum number $\lambda(S)$ of
$S$-trees and then $S$ runs over all $k$-subsets of $V(G)$ to get
the minimum value of $\lambda(S)$.

From a theoretical perspective, both extremes of the generalized
edge-connectivity problem are fundamental theorems in combinatorics.
One extreme is when we have two terminals. In this case
edge-disjoint trees are just edge-disjoint paths between the two
terminals, and so the problem becomes the well-known Menger theorem.
The other extreme is when all the vertices are terminals. In this
case edge-disjoint trees are just spanning trees of the graph, and
so the problem becomes the classical Nash-Williams-Tutte theorem,
see \cite{Nash,Tutte}.

Product graphs are an important method to construct large graphs from small ones, so it has many applications in the design and analysis of networks, see \cite{Grotschel1,Ku,LXZW}. The lexicographic product, together with the Cartesian product, the strong product and the direct product, is the main four standard products of graphs. More information about the (edge) connectivity of these four product graphs can be found in \cite{LLSun,WS,Hammack,Klavzar,Bresar,Feng,XY}.

In this paper, we study the generalized edge-connectivity of the lexicographic product graph and get the following theorems.

\begin{thm}\label{thm2}
Let $G$ and $H$ be two non-trivial graphs and $G$ is connected. Then $\lambda_3(G\circ H)\geq \lambda_3(H)+\lambda_3(G)|V(H)|$. Moreover, the lower bound is sharp.
\end{thm}

\begin{thm}\label{thm3}
Let $G$ and $H$ be two non-trivial graphs and $G$ is connected. Then
$$
\lambda_3(G\circ H)\leq \min \Big \{\Big \lfloor
\frac{4\lambda_3(G)+2}{3}\Big \rfloor
|V(H)|^2,\delta(H)+\delta(G)|V(H)|\Big \}.
$$
Moreover, the upper bound is sharp.
\end{thm}

\section{Preliminaries}

Let $G$ be a graph and $S\subseteq V(G)$. Let $G[S]$ denote the induced subgraph of $G$ on the vertex set $S$ and let $d_G(v)$ denote the degree of $v$ in $G$, where $v \in V(G)$. If $u$ and $v$ are two vertices on a path $P$, $uPv$ will denote the segment of $P$ from $u$ to $v$. Given sets $X$, $Y$ of vertices, we call a path $P$ an $XY$-$path$ if the end vertices of $P$ are in $X$ and $Y$, respectively, and all inner vertices are in neither $X$ nor $Y$. Two distinct paths are \emph{edge disjoint} if they have no edges in common. Two distinct paths are \emph{internally disjoint} if they have no internal vertices in common. Two distinct paths are \emph{vertex disjoint} if they have no vertices in common. For $X=\{x_1,x_2,\cdots,x_k\}$ and $Y=\{y_1,y_2,\cdots,y_k\}$, an $XY$-$linkage$ is defined as a set $Q$ of $k$ vertex-disjoint paths $x_iP_iy_i$, $1\leq i\leq k$.

Let $G=(V_1,E_1)$, $H=(V_2,E_2)$, the lexicographic product $G\circ H$ of $G$ and $H$ is defined as follows: $V(G\circ H)=V_1\times V_2$, two vertices $(u,v)$ and $(u',v')$ are adjacent if and only if either $uu'\in E_1$ or $u=u'$, $vv'\in E_2$. On other words, $G\circ H$ is obtained by substituting a copy $H(u)$ of $H$ for every vertex $u$ of $G$ and joining all vertices of $H(u)$ with all vertices of $H(u')$ if $uu'\in E_1$. Unlike the other product, the lexicographic product does not satisfy the commutative law, that is, $G\circ H$ need not be isomorphic to $H\circ G$. By a simple observation, $G\circ H$ is connected if and only if $G$ is connected. Moreover, $\delta(G\circ H)=\delta(G)|V(H)|+\delta(H)$. The edge $(u,v)$$(u',v')$ is called one-type edge if $uu'\in E_1$ and $v=v'$; two-type edge if $vv'\in E_2$ and $u=u'$; three-type edge if $uu'\in E_1$ and $v\neq v'$.

The vertex set $G(v)=\{(u,v)|u\in V_1\}$ for some fixed vertex $v$ of $H$ is called a layer of graph $G$ or simply a $G$-$layer$. Analogously we define the $H$-$layer$ with respect to a vertex $u$ of $G$ and denote it by $H(u)$. It is not hard to see that any $G$-$layer$ induces a subgraph of $G\circ H$ that is isomorphic to
$G$ and any $H$-$layer$ induces a subgraph of $G\circ H$ that is isomorphic to $H$. For a subset $W$ of $V(G)$ with $W=\{u_1,\cdots,u_t\}$, $H(W)=H(u_1)\cup \cdots \cup H(u_t)$. $K_{u_1,\cdots,u_t}$ denotes a subgraph of $G\circ H$, where $V(K_{u_1,\cdots,u_t})=V(W\circ H)$, $E(K_{u_1,\cdots,u_t})=E(G[u_1,\cdots,u_t]\circ H)\setminus E(H(W))$, namely, the end vertices of an edge of $K_{u_1,\cdots,u_t}$ are in different $H$-layers.

Let $G$ be a connected graph, $S=\{x,y,z\}\subseteq V(G)$, and $T$ be an $S$-tree. By deleting some vertices and edges of $T$, it is easy to check that $T$ has exactly two types, one is called type $\uppercase\expandafter{\romannumeral1}$ if $T$ is just a path whose two end vertices belong to $S=\{x,y,z\}$; the other is called type $\uppercase\expandafter{\romannumeral2}$ if it is a tree with exactly three leaves $x$, $y$, $z$. Note that the vertices in a tree of type $\uppercase\expandafter{\romannumeral1}$ have degree two except the two end vertices in $S$. If $T$ is of type $\uppercase\expandafter{\romannumeral2}$, every vertex in $T\setminus S$ has degree two except one vertex of degree three. In this paper, we assume that each $S$-tree is of type $\uppercase\expandafter{\romannumeral1}$ or $\uppercase\expandafter{\romannumeral2}$.

\begin{pro}\label{k-2}
Let $G$ be a graph with $\lambda_3(G)=k\geq2$, $S=\{x,y,z\}\subseteq V(G)$. Then there exist $k-2$ edge-disjoint $S$-trees $T_1,\cdots,T_k$ such that $E(T_i)\cap E(G[S])=\emptyset$.
\end{pro}
\begin{proof}
By the definition of $S$-trees, we know that  $|E(T_i)\cap E(G[S])|\leq2$ and $|\{T_i\,|\,E(T_i)\cap E(G[S])\neq\emptyset\}|\leq3$. Let $\{T_1,\cdots,T_k\}$ be $k$ edge-disjoint $S$-trees. If $|\{T_i\,|\,E(T_i)\cap E(G[S])\neq\emptyset\}|\leq2$, we are done. Thus, suppose $|\{T_i\,|\,E(T_i)\cap E(G[S])\neq\emptyset\}|=3$. Without loss of generality, assume $E(T_i)\cap E(G[S])\neq\emptyset$, where $i=1,2,3$. Then $T_1$, $T_2$, $T_3$ have the structures $F_1$ or $F_2$ as shown in Figure 1. For these two cases, we can obtain $T'_1$, $T'_2$, $T'_3$ from $T_1$, $T_2$, $T_3$ such that $E(T'_1)\cap E(G[S])=\emptyset$. See Figure $F'_1$ and $F'_2$, where the tree $T'_1$ is shown by gray lines. Thus $T'_1,T_4,\cdots,T_k$ are our desired $S$-trees.
\end{proof}

\begin{figure}[h,t,b,p]
\begin{center}
\scalebox{0.7}[0.7]{\includegraphics{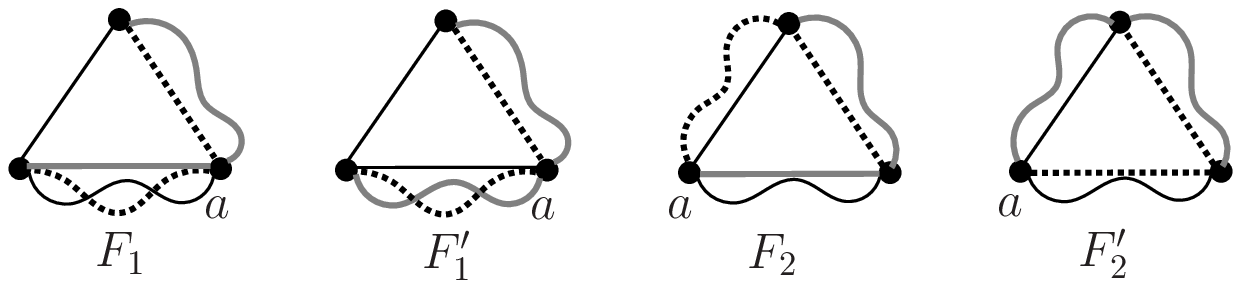}}\\[20pt]

Figure~1. Three $S$-trees of type $\uppercase\expandafter{\romannumeral1}$.
\end{center}
\end{figure}

Li et al. \cite{LLMS,LM4} got the following results which are useful for our proof.

\begin{pro} \cite {LLMS} \label{lambdalambda}
For any graph $G$ of order $n$, $\lambda_k(G)\leq \lambda(G)$. Moreover, the upper bound is
tight.
\end{pro}

\begin{obs}\cite {LLMS} \label{lambdadelta}
If G be a connected graph, then $\lambda_k(G)\leq \delta(G)$. Moreover, the upper bound is
tight.
\end{obs}

\begin{pro} \cite{LM4}\label{delta-1}
Let $G$ be a connected graph of order $n$ with minimum degree
$\delta$. If there are two adjacent vertices of degree $\delta$,
then $\lambda_k(G)\leq \delta-1$ for $3\leq k\leq n$. Moreover, the
upper bound is sharp.
\end{pro}

From Proposition \ref{delta-1}, it is easy to get the following observation.

\begin{obs}\label{k+1}
Let $G$ be a connected graph with $\lambda_3(G)=k$, $x$ and $y$ be two adjacent vertices of $G$. Then $d_G(x)\geq k+1$ or $d_G(y)\geq k+1$.
\end{obs}

Before getting into our main results, we give an elementary observation.

\begin{obs}\label{obs2}
$(i)$~Let $G$ and $H$ be two non-trivial graphs and $G$ is connected, let $x$, $y$, $z$ be three distinct vertices of $H$ and $T_1,T_2,\cdots, T_{k}$ be $k$ edge-disjoint $\{x,y,z\}$-trees in $H$. Then $G\circ\bigcup_{i=1}^k T_i=\bigcup_{i=1}^k (G\circ T_i)$. Moreover if $V(T_i)\cap V(T_j)=W$ for $i\neq j$, then $E(G\circ T_i)\cap E(G\circ T_j)=E(G\circ W)\setminus E(W(G))$.

$(ii)$~Let $G$ and $H$ be two non-trivial graphs and $G$ is connected, let $x$, $y$, $z$ be three distinct vertices of $G$ and $T_1,T_2,\cdots, T_{k}$ be $k$ edge-disjoint $\{x,y,z\}$-trees in $G$. Then $\bigcup_{i=1}^k T_i\circ H=\bigcup_{i=1}^k (T_i\circ H)$. Moreover if $V(T_i)\cap V(T_j)=W$  for $i\neq j$, $E(T_i\circ H)\cap E(T_j\circ H)=E(H(W))$.
\end{obs}

For the above observation, we give two examples.

\begin{example}
Let $G$ be a complete graph of order 4 and $H$ be an arbitrary graph. The structure of $G\circ (T_1\cup T_2)$ is shown as $F_a$ in Figure 2.
\end{example}

\begin{example}
Let $G$ be a path of length 2 and $H$ be a complete graph of order 4. The structure of $(T_1\cup T_2)\circ H$ is shown as $F_b$ in Figure 2.
\end{example}

\begin{figure}[h,t,b,p]
\begin{center}
\scalebox{0.8}[0.8]{\includegraphics{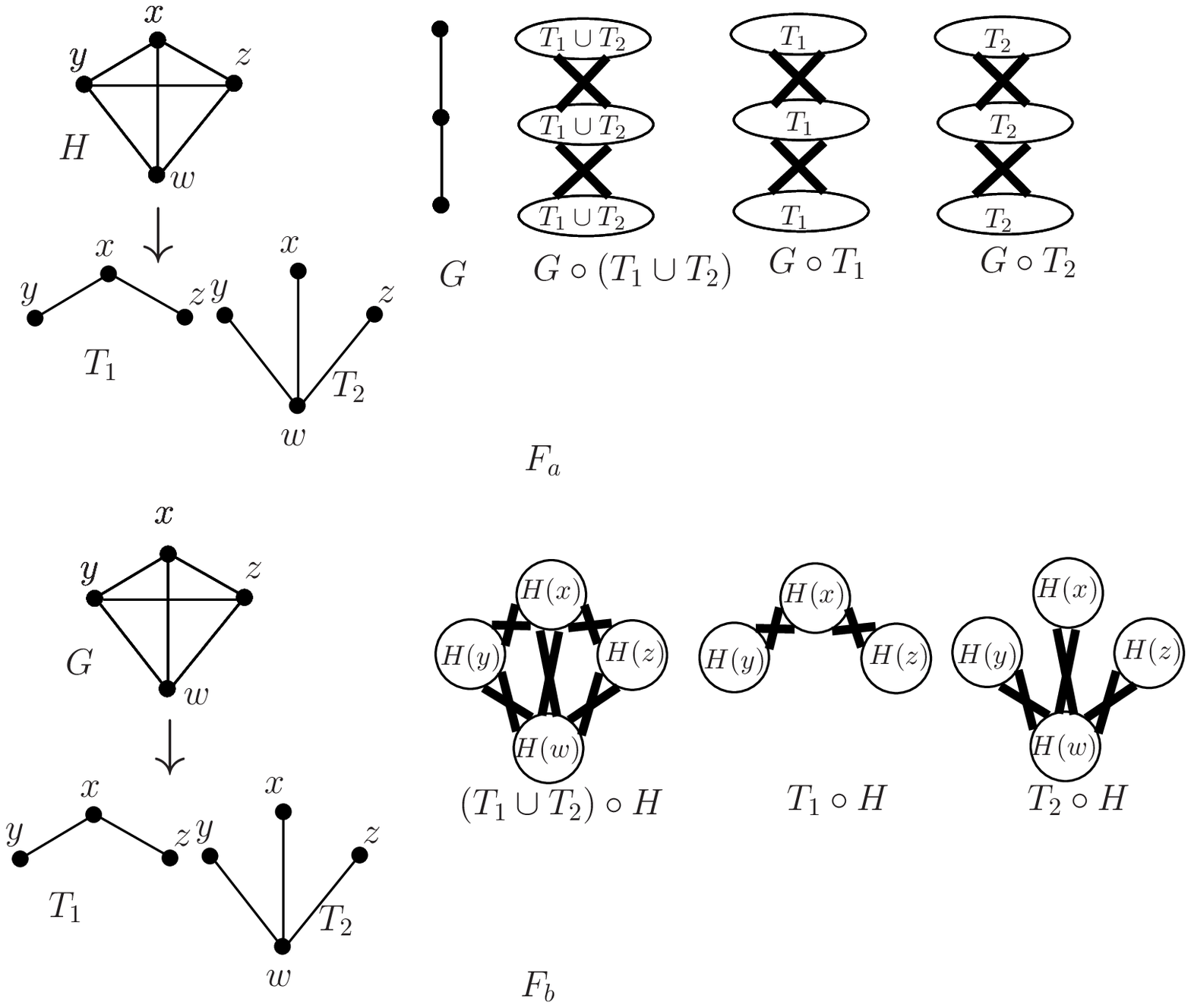}}\\[20pt]

Figure~2. The structures of $G\circ (T_1\cup T_2)$ and $(T_1\cup T_2)\circ H$.
\end{center}
\end{figure}

\section{Lower bound of $\lambda_3(G\circ H)$}
In this section, we give the lower bound of generalized 3-edge-connectivity of the lexicographic product of two graphs. Before proceeding, we give some notations and lemmas.

Set $V(G)=\{u_1,u_2,\cdots,u_{n_1}\}$, $V(H)=\{v_1,v_2,\cdots,v_{n_2}\}$ and set $\lambda_3(G)=\ell_1$, $\lambda_3(H)=\ell_2$ for simplicity. Let $S=\{x,y,z\}\subseteq V(G\circ H)$. In total, we construct our desired $S$-trees on two stages: $\ell_2$ edge-disjoint $S$-trees by one-type and two-type edges on Stage $\uppercase\expandafter{\romannumeral1}$ and $\ell_1n_2$ edge-disjoint $S$-trees by one-type and three-type edges on Stage $\uppercase\expandafter{\romannumeral2}$. If $H$ is disconnected, then $\lambda_3(H)=\ell_2=0$ as defined, thus we omit Stage $\uppercase\expandafter{\romannumeral1}$ immediately. Next we always assume $H$ is connected.

According to the position of $x$, $y$, $z$ in $G\circ H$, we give some lemmas as follows.

\begin{lem}\label{lem1}
If $x,y,z$ belong to the same $H(u_i)$, $1\leq i\leq n$, then there exist $\ell_2+\ell_1n_2$ edge-disjoint $S$-trees.
\end{lem}

\begin{proof}
Without loss of generality, assume $x,y,z\in H(u_1)$ and $x=(u_1,v_1)$, $y=(u_1,v_2)$, $z=(u_1,v_3)$. On Stage $\uppercase\expandafter{\romannumeral1}$, there are $\ell_2$ edge-disjoint $S$-trees in $H(u_1)$, since $\lambda_3(H)=\ell_2$. On Stage $\uppercase\expandafter{\romannumeral2}$, since $\lambda_3(G)=\ell_1$ and Observation \ref{lambdadelta}, there are $\ell_1$ neighbors of $u_1$ in $G$, say $\beta_1,\beta_2,\cdots,\beta_{\ell_1}$. Thus $T^*_{ij}=x(\beta_{i},v_j)\cup y(\beta_{i},v_j)\cup z(\beta_{i},v_j)$($1\leq i\leq \ell_1$ and $1\leq j\leq n_2$) are $\ell_1n_2$ $S$-trees. By Observation \ref{obs2}, it is easy to see that these $\ell_2+\ell_1n_2$ $S$-trees are edge-disjoint, as desired.
\end{proof}

\begin{lem}\label{lem2}
If only two of $\{x,y,z\}$ belong to the same $H(u_i)$, $1\leq i\leq n$, then there exist $\ell_2+\ell_1n_2$ edge-disjoint $S$-trees.
\end{lem}

\begin{proof}
Suppose $x,y\in H(u_1)$, $z\in H(u_2)$. Let $x''$, $y''$ be the vertices in $H(u_2)$ corresponding to $x$, $y$, respectively, and $z'$ be the vertex in $H(u_1)$ corresponding to $z$. Consider the following two cases.

{\bf Case 1.}~~ $z'\in \{x,y\}$.

Without loss of generality, assume $z'=x$ and $x=(u_1,v_1)$, $y=(u_1,v_2)$, $z=(u_2,v_1)$.

Since $\lambda(H)\geq \lambda_3(H)=\ell_2$, there are $\ell_2$ edge-disjoint $v_1v_2$-paths $P_1,P_2,\cdots, P_{\ell_2}$ in $H$ such that $\ell(P_1)\leq \ell(P_2)\leq \cdots \leq \ell(P_{\ell_2})$. For $1\leq i\leq \ell_2$, denote the neighbor of $v_1$ in $P_i$ by $\alpha_i$. Notice that $\alpha_p\neq \alpha_q$ for $p\neq q$, $1\leq p, q\leq \ell_2$.

Since $\lambda(G)\geq \lambda_3(G)=\ell_1$, there exist $\ell_1$ edge-disjoint $u_1u_2$-paths $Q_1,Q_2,\cdots, Q_{\ell_1}$ in $G$ such that $\ell(Q_1)\leq \ell(Q_2)\leq \cdots \leq \ell(Q_{\ell_1})$. For each $i$ with $1\leq i\leq \ell_1$, set $Q_i=u_1\beta_{i,1}\beta_{i,2}\cdots\beta_{i,t_i-1}u_2$ and $\ell(Q_i)=t_i$. Also, note that $\beta_{p,1}\neq \beta_{q,1}$ for $p\neq q$, $1\leq p,q\leq \ell_1$.

Firstly, we come to Stage $\uppercase\expandafter{\romannumeral1}$. Choose the longest $u_1u_2$-path $Q_{\ell_1}$ and construct our desired $\ell_2$ $S$-trees according to $Q_{\ell_1}$. If $v_1$ and $v_2$ are not adjacent in $H$, then let $T^*_i=P_i(u_1)\cup Q_{\ell_1}(\alpha_i)\cup z(u_2,\alpha_i)$ for $1\leq i\leq \ell_2$, where $P_i(u_1)$ is the path in $H(u_1)$ corresponding to $P_i$, $Q_{\ell_1}(\alpha_i)$ is the path in $G(\alpha_i)$ corresponding to $Q_{\ell_1}$.

So suppose $v_1$ and $v_2$ are adjacent in $H$, that is, $P_1=v_1v_2$ and $(u_1,\alpha_1)=y$. Since $\lambda_3(H)=\ell_2$ and Observation \ref{k+1}, it follows that, $d_{H}(v_1)\geq \ell_2+1$ or $d_{H}(v_2)\geq \ell_2+1$. Let $d_{H}(v_1)\geq \ell_2+1$ (the case that $d_{H}(v_2)\geq \ell_2+1$ can be proved similarly). For $P_1$, choose another neighbor $\alpha_{\ell_2+1}$ of $v_1$ in $H$, which is not $\alpha_i$ and $v_2$ ($2\leq i\leq \ell_2$). Let $T^*_1=xy\cup x(u_1,\alpha_{\ell_2+1})\cup Q_{\ell_1}(\alpha_{\ell_2+1})\cup z(u_2,\alpha_{\ell_2+1})$, where $Q_{\ell_1}(\alpha_{\ell_2+1})$ is the path in $G(\alpha_{\ell_2+1})$ corresponding to $Q_{\ell_1}$. For $P_i$ with $2\leq i\leq \ell_2$, set $T^*_i=P_i(u_1)\cup Q_{\ell_1}(\alpha_i)\cup z(u_2,\alpha_i)$, where $P_i(u_1)$ is the path in $H(u_1)$ corresponding to $P_i$, $Q_{\ell_1}(\alpha_i)$ is the path in $G(\alpha_i)$ corresponding to $Q_{\ell_1}$. Thus, by $(i)$ of Observation \ref{obs2}, these $\ell_2$ $S$-trees are edge-disjoint.

Note that, on Stage $\uppercase\expandafter{\romannumeral1}$, if $v_1$ and $v_2$ are adjacent in $H$, then we choose another neighbor of $v_1$ rather than $v_2$ (or another neighbor of $v_2$ rather than $v_1$) for the path $P_1$. The aim is to make sure that the one-type edges incident to $x$ in $G(v_1)$ and to $y$ in $G(v_2)$ corresponding to each $Q_i$ remain to be used on Stage $\uppercase\expandafter{\romannumeral2}$.

On Stage $\uppercase\expandafter{\romannumeral2}$, we construct $\ell_1n_2$ $S$-trees corresponding to the length of $Q_i$ in non-decreasing order. We distinguish two subcases by the length of $Q_1$.

{\bf Subcase 1.1.}~~$t_1\geq 2$.

For $Q_1$, we find $n_2$ internally disjoint $xy$-paths $A_1,A_2,\cdots, A_{n_2}$ in $K_{u_1,\beta_{1,1}}$ by using the remaining edges after Stage $\uppercase\expandafter{\romannumeral1}$, and get a $V(H(\beta_{1,1}))V(H(\beta_{1,t_1-1}))$-linkage $B_1,B_2,\cdots, B_{n_2}$ by the three-type edges according to $\beta_{1,1}Q_1\beta_{1,t_1-1}$. Thus $T^*_i=A_i\cup B_i\cup (\beta_{1,t_1-1},v_i)z$ are $n_2$ edge-disjoint $S$-trees, where $1\leq i\leq n_2$ and the subscript $i$ of $v_i$ is expressed module $n_2$ as one of $1,2,\cdots,n_2$.

\begin{figure}[h,t,b,p]
\begin{center}
\scalebox{0.8}[0.8]{\includegraphics{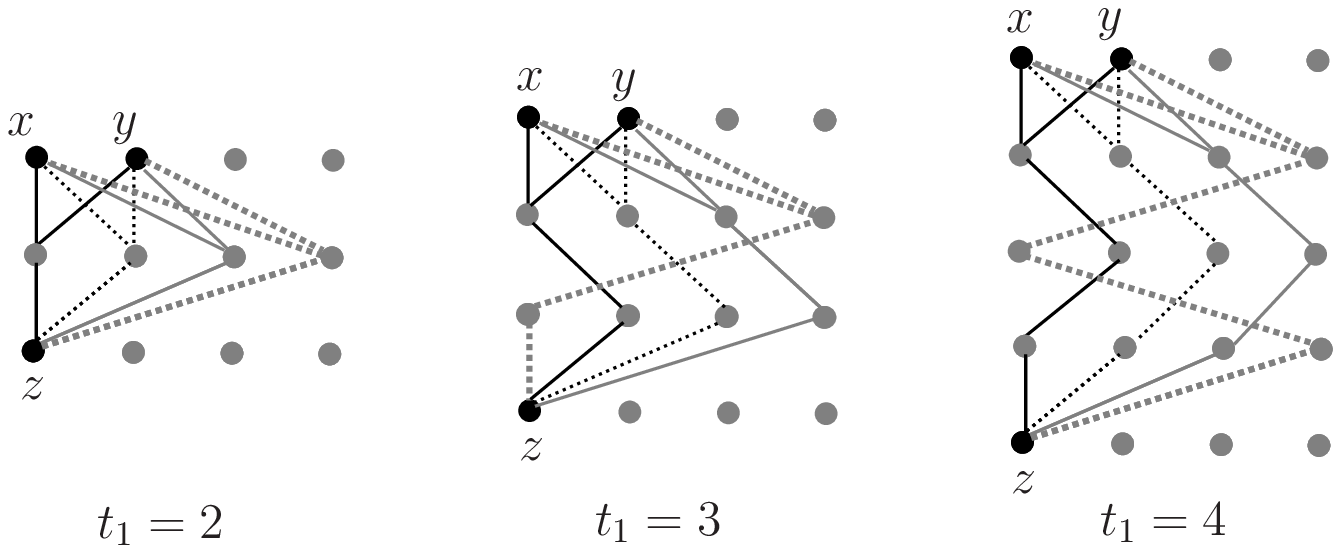}}\\[20pt]
Figure~3. The $4$ edge-disjoint $S$-trees corresponding to $Q_1$ when $n_2=4$\\
~~~~(The edges of a tree are shown by the same type of lines).
\end{center}
\end{figure}

Indeed, this can always be done. Set $A_i=x(\beta_{1,1},v_i)y$ for $1\leq i\leq n_2$. If $t_1=2$, then $B_i=\emptyset$. If $t_1\geq3$, that is, $d_{Q_1}(\beta_{1,1},\beta_{1,t_1-1})=t_1-2\geq 1$, then $B_i$ has the structure as follows. If $t_1$ is even, then let $B_i=(\beta_{1,1},v_i)(\beta_{1,2},v_{i+1})(\beta_{1,3},v_i)(\beta_{1,4},v_{i+1})\cdots(\beta_{1,t_1-1},v_i)$ and $T^*_i=A_i\cup B_i\cup (\beta_{1,t_1-1},v_i)z$. If $t_1$ is odd, then let $B_i=(\beta_{1,1},v_i)(\beta_{1,2},v_{i+1})(\beta_{1,3},v_i)\\(\beta_{1,4},v_{i+1})\cdots (\beta_{1,t_1-1},v_{i+1})$ and $T^*_i=A_i\cup B_i\cup (\beta_{1,t_1-1},v_{i+1})z$. Take for example, let $n_2=4$, then 4 edge-disjoint $S$-trees are shown in Figure 3 when $t_1=2$, $t_1=3$ and $t_1=4$, respectively.

Similar to these $n_2$ $S$-trees corresponding to $Q_1$, we continue to construct $n_2$ $S$-trees corresponding to $Q_i$ by the edges in accord with $E(Q_i)$, since $\ell(Q_i)\geq2$ for each $i$, where $2\leq i\leq \ell_1$.

{\bf Subcase 1.2.}~~ $t_1=1$, that is, $Q_1=u_1u_2$.

Since $\lambda_3(G)=\ell_1$, it follows by Observation \ref{k+1} that $d_{G}(u_1)\geq \ell_1+1$ or $d_{G}(u_2)\geq \ell_1+1$.

If $d_{G}(u_1)\geq \ell_1+1$, then denote another neighbor of $u_1$ in $G$ by $\beta_{\ell_1+1,1}$ except $u_2$ and $\beta_{i,1}$ ($2\leq i\leq \ell_1$). For $Q_1$, we find out $n_2$ edge-disjoint $S$-trees as follows. Let $T^*_1=(\beta_{\ell_1+1,1},v_1)x\cup (\beta_{\ell_1+1,1},v_1)y\cup xz$, $T^*_2=(\beta_{\ell_1+1,1},v_2)x\cup (\beta_{\ell_1+1,1},v_2)y\cup yz$, $T^*_i=(u_2,v_i)x\cup (u_2,v_i)y\cup (u_2,v_i)(u_1,v_{i+1})\cup (u_1,v_{i+1})z$ for $3\leq i\leq n_2-1$, $T^*_{n_2}=(u_2,v_{n_2})x\cup (u_2,v_{n_2})y\cup (u_2,v_{n_2})(u_1,v_3)\cup (u_1,v_3)z$. See Figure 4($a$).

\begin{figure}[h,t,b,p]
\begin{center}
\scalebox{0.9}[0.9]{\includegraphics{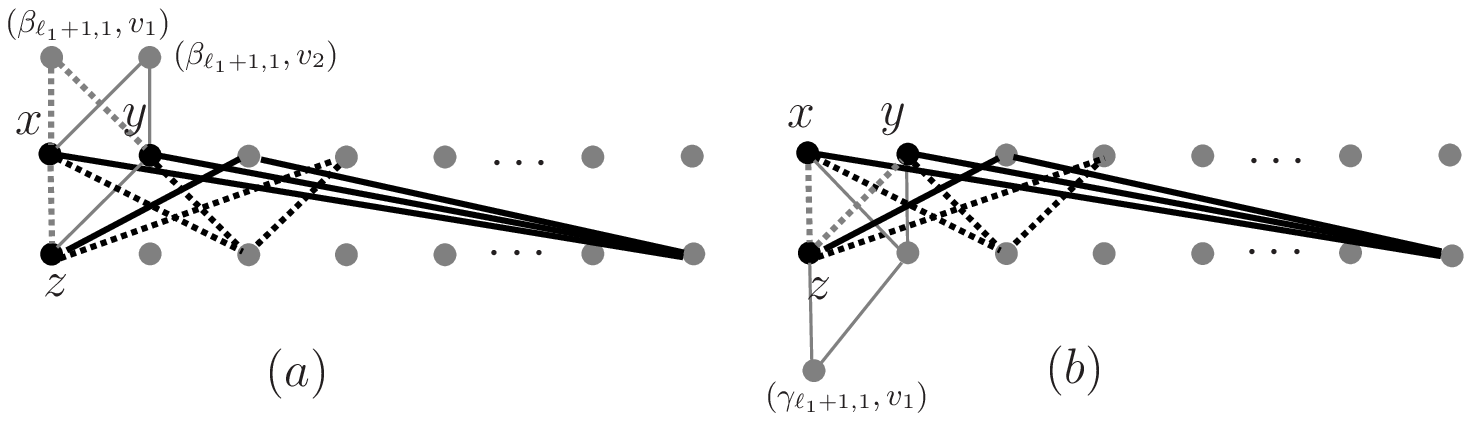}}\\[20pt]

Figure~4. The $n_2$ edge-disjoint $S$-trees for $Q_1$ corresponding to $Q_1=u_1u_2$\\~~~~ (The edges of a tree are shown by the same type of lines).
\end{center}
\end{figure}

If $d_{G}(u_2)\geq \ell_1+1$, then denote another neighbor of $u_2$ in $G$ by $\gamma_{\ell_1+1}$ except $u_1$ and $\beta_{i,t_i-1}$ ($2\leq i\leq \ell_1$). For $Q_1$, set $T^*_1=xz\cup zy$, $T^*_2=xy''\cup y''y\cup (\gamma_{\ell_1+1},v_1)y''\cup (\gamma_{\ell_1+1},v_1)z$, $T^*_i=(u_2,v_i)x\cup (u_2,v_i)y\cup (u_2,v_i)(u_1,v_{i+1})\cup (u_1,v_{i+1})z$ for $3\leq i\leq n_2-1$, $T^*_{n_2}=(u_2,v_{n_2})x\cup (u_2,v_{n_2})y\cup (u_2,v_{n_2})(u_1,v_3)\cup (u_1,v_3)z$. See Figure 4($b$).

Corresponding to $Q_i$ with $2\leq i\leq \ell_1$, construct $n_2$ edge-disjoint $S$-trees similar to that in Subcase 1.1 of Stage $\uppercase\expandafter{\romannumeral2}$.

Since the edges on Stage $\uppercase\expandafter{\romannumeral2}$ are of three-type corresponding to each $Q_i$ besides three one-type edges incident to $x$, $y$ and $z$ that are not used on Stage $\uppercase\expandafter{\romannumeral1}$, it follows that the edges used on Stage $\uppercase\expandafter{\romannumeral2}$ are different from those used on Stage $\uppercase\expandafter{\romannumeral1}$. And by $(ii)$ of Observation \ref{obs2}, these $\ell_1n_2$ $S$-trees on Stage $\uppercase\expandafter{\romannumeral2}$ are edge-disjoint, as desired.

{\bf Case 2.}~~ $z'\notin \{x,y\}$.

Assume $x=(u_1,v_1)$, $y=(u_1,v_2)$, $z=(u_2,v_3)$. Let $S'=\{v_1,v_2,v_3\}$, $S''=\{x,y,z'\}$.

Since $\lambda(G)\geq \lambda_3(G)=\ell_1$, there exist $\ell_1$ edge-disjoint $u_1u_2$-paths $Q_1,Q_2,\cdots, Q_{\ell_1}$ in $G$ such that $\ell(Q_1)\leq \ell(Q_2)\leq \cdots \leq \ell(Q_{\ell_1})$.

Since $\lambda_3(H)=\ell_2$, there are $\ell_2$ edge-disjoint
$S'$-trees $T_1,T_2,\cdots, T_{\ell_2}$ in $H$. Recall that $0\leq
|E(T_i)\cap E(G[S'])|\leq 2$. By Proposition \ref{k-2}, suppose
$E(T_i)\cap E(G[S'])=\emptyset$ for $3\leq i\leq \ell_2$. According
to whether $T_1$ and $T_2$ have edges in $E(G[S'])$ or not, $T_1$
and $T_2$ have one of the following structures.

{\bf Subcase 2.1.} $E(T_1)\cap E(G[S'])=\emptyset$ and $E(T_2)\cap E(G[S'])=\emptyset$.

For $1\leq i\leq \ell_2$, denote the neighbor of $v_3$ in $T_i$ by $\alpha_i$.

On Stage $\uppercase\expandafter{\romannumeral1}$, let $T^*_i=T_i(u_1)\cup Q_{\ell_1}(\alpha_i)\cup z(u_2,\alpha_i)$ ( where $1\leq i\leq \ell_2$, and $T_i(u_1)$ is the path in $H(u_1)$ corresponding to $T_i$, $Q_{\ell_1}(\alpha_i)$ is the path in $G(\alpha_i)$ corresponding to $Q_{\ell_1}$).

On Stage $\uppercase\expandafter{\romannumeral2}$, if $\ell(Q_i)\geq 2$ for each $i$ with $1\leq i\leq \ell_1$, construct $n_2$ $S$-trees similar to Case 1; otherwise $\ell(Q_1)=1$, then either $u_1$ or $u_2$ has a neighbor which is not on each $u_1u_2$-paths $Q_i$ in $G(v_1)$. Then $n_2$ $S$-trees corresponding to $Q_1$ are shown in Figure 5 and construct $n_2$ $S$-trees similar to Case 1 for each $i$ with $2\leq i\leq \ell_1$.

\begin{figure}[h,t,b,p]
\begin{center}
\scalebox{0.9}[0.9]{\includegraphics{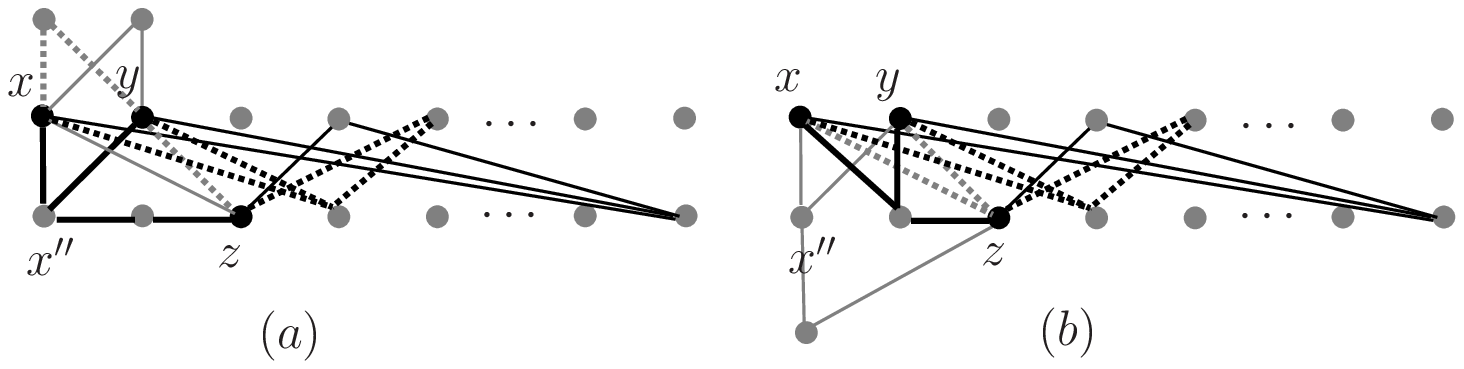}}\\[20pt]

Figure~5. The $n_2$ edge-disjoint $S$-trees on Stage $\uppercase\expandafter{\romannumeral2}$ for $Q_1$\\~~~~(the edges of a tree are shown by the same type of lines).
\end{center}
\end{figure}

\begin{figure}[h,t,b,p]
\begin{center}
\scalebox{0.7}[0.7]{\includegraphics{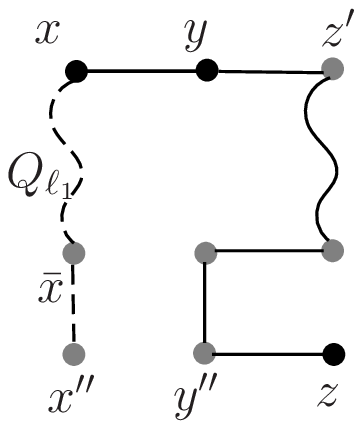}}\\[20pt]

Figure~6. The solid lines stand for the edges of the $S$-tree.
\end{center}
\end{figure}

{\bf Subcase 2.2.} $E(T_1)\cap E(G[S'])\neq\emptyset$ and $E(T_2)\cap E(G[S'])=\emptyset$.

Suppose $|E(T_1)\cap E(G[S'])|=1$ and $E(T_2)\cap E(G[S'])=\emptyset$. Without loss of generality, suppose $E(T_1)\cap E(G[S'])=v_1v_2$. For $1\leq i\leq \ell_2$, denote the neighbor of $v_3$ in $T_i$ by $\alpha_i$. Construct $\ell_2+\ell_1n_2$ $S$-trees similar to Subcase 2.1 by making use of $\alpha_i$. It remains to consider $|E(T_1)\cap E(G[S'])|=2$ and $E(T_2)\cap E(G[S'])=\emptyset$. Without loss of generality, suppose $E(T_1)\cap E(G[S])=\{v_1v_2,v_2v_3\}$. For $T_1$, if $d_{Q_{\ell_1}}(u_1,u_2)\geq 2$, then $T^*_1$ has the structure as shown in Figure 6, where $\bar{x}$ is the neighbor of $x''$ in $Q_{\ell_1}(v_1)$; if $d_{Q_{\ell_1}}(u_1,u_2)=1$, set $T^*_1=xyz'z$. Construct other $\ell_2+\ell_1n_2-1$ $S$-trees similar to Subcase 2.1. Thus there exist $\ell_2+\ell_1n_2$ $S$-trees.

\begin{figure}[h,t,b,p]
\begin{center}
\scalebox{0.7}[0.7]{\includegraphics{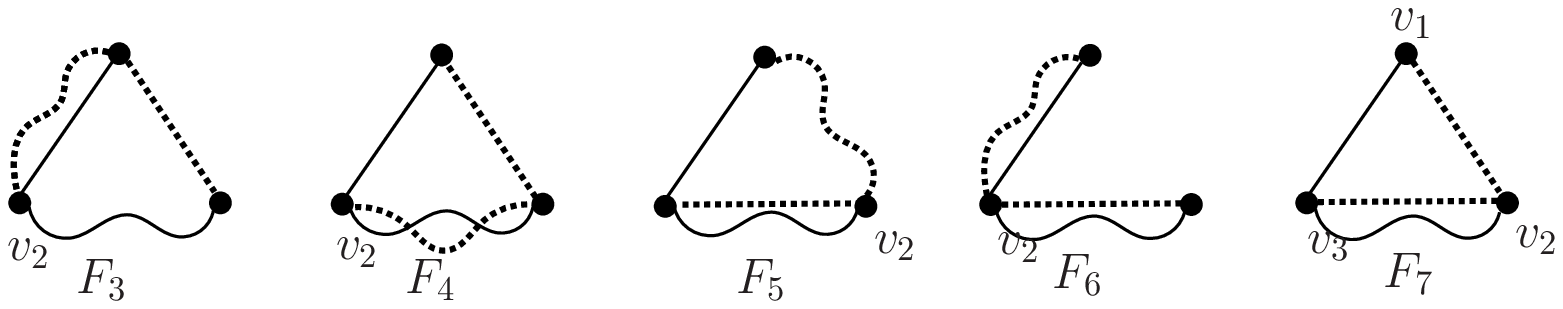}}\\[20pt]

Figure~7. Two $S'$-trees of type
$\uppercase\expandafter{\romannumeral1}$ in Case 2 of Lemma
\ref{lem2}.
\end{center}
\end{figure}

{\bf Subcase 2.3.} $E(T_1)\cap E(G[S'])\neq\emptyset$ and $E(T_2)\cap E(G[S'])\neq\emptyset$.

Without loss of generality, suppose $|E(T_2)\cap E(G[S])|=1$. If $|E(T_1)\cap E(G[S'])|=1$, then we may assume that the trees $T_1$ and $T_2$ have one of the structures  $F_3,F_4,F_5,F_6$ as shown in Figure 7. For $1\leq i\leq
\ell_2$, denote the neighbor of $v_2$ in $T_i\setminus \{v_1,v_3\}$
by $\alpha_i$. Construct $\ell_2+\ell_1n_2$ $S$-trees similar to
Subcase 2.1.

If $|E(T_1)\cap E(G[S'])|=2$, then the trees $T_1$ and $T_2$ have the structure $F_7$ as
shown in Figure 7, where $T_1$ is shown by dotted lines. Construct $T^*_1$ as shown in Figure 6. For $2\leq i\leq \ell_2$, denote the neighbor of $v_2$ in $T_i\setminus \{v_1,v_3\}$ by $\alpha_i$. Construct $\ell_2+\ell_1n_2-1$ $S$-trees similar to Subcase 2.1. Thus, there exist $\ell_2+\ell_1n_2$ $S$-trees.

By Observation \ref{obs2}, these $\ell_2+\ell_1n_2$ $S$-trees are edge-disjoint in each case, as desired.
\end{proof}

\begin{lem}\label{lem9}
If $x,y,z$ belong to distinct $H(u_i)$s, then there exist $\ell_2+\ell_1n_2$ edge-disjoint $S$-trees.
\end{lem}
\begin{proof}
Assume $x\in H(u_1)$, $y\in H(u_2)$, $z\in H(u_3)$. Let $y'$, $z'$ be the vertex corresponding to $y$, $z$ in $H(u_1)$, $x''$, $z''$ be the vertex corresponding to $x$, $z$ in $H(u_2)$, $x'''$, $y'''$ be the vertex corresponding to $x$, $y$ in $H(u_3)$. We distinguish the following three cases.

{\bf Case 1.}~~ $x$, $y'$, $z'$ are the same vertex in $H(u_1)$.

We may assume $x=(u_1,v_1)$, $y=(u_2,v_1)$, $z=(u_3,v_1)$. Since $\lambda_3(H)=\ell_2$, there are $\ell_2$ neighbors of $v_1$ in $H$, say $\alpha_1,\alpha_2,\cdots,\alpha_{\ell_2}$. Since $\lambda_3(G)=\ell_1$, there are $\ell_1$ edge-disjoint $\{u_1,u_2,u_3\}$-trees $T_1,T_2,\cdots, T_{\ell_1}$ in $G$. For a tree $T_i$ in $G$, set by $T_i(\alpha_j)$ the tree in $G(\alpha_j)$ corresponding to $T_i$ for $1\leq i\leq \ell_1$, $1\leq j\leq \ell_2$.

On Stage $\uppercase\expandafter{\romannumeral1}$, for $1\leq j\leq \ell_2$, set $T_j^{*}=T_1(\alpha_j)\cup x(u_1,\alpha_j)\cup y(u_2,\alpha_j)\cup z(u_3,\alpha_j)$.

On Stage $\uppercase\expandafter{\romannumeral2}$, for each $j$ with $1\leq j\leq \ell_1$, if $T_j$ is of type $\uppercase\expandafter{\romannumeral1}$, we may assume $d_{T_j}(u_2)=2$. Denote the neighbor of $u_1$, $u_3$ in $T_j$ by $\eta_{j}$, $\gamma_{j}$ and the neighbor of $u_2$ by $\beta_{j}$, $\bar{\beta_{j}}$ ($\beta_{j}$ is nearer to $u_1$ than $\bar{\beta_{j}}$), where $\beta_{j}$, $\eta_{j}$ may be the same vertex, $\bar{\beta_{j}}$, $\gamma_{j}$ may be the same vertex. Corresponding to $u_1T_ju_2$ and $u_2T_ju_3$, we find $n_2$ edge-disjoint $xy$-paths $A=\{A_1,\cdots,A_{n_2}\}$ and edge-disjoint $yz$-paths $B=\{B_1,\cdots,B_{n_2}\}$ respectively. Then $T^*_{ij}=A_i\cup B_i$ ($1\leq i\leq {n_2}$) are $n_2$ edge-disjoint $S$-trees. Since the construction of $B$ is similar to that of $A$, we only provide the construction of $A$ according to $d_{T_j}(u_1,u_2)$. If $d_{T_j}(u_1,u_2)=1$, set $A_1=xy$, $A_i=x(u_2,v_i)\cup (u_2,v_i)(u_1,v_{i+1})\cup (u_1,v_{i+1})y$ for $2\leq i\leq n_2-1$, $A_{n_2}=x(u_2,v_{n_2})\cup (u_2,v_{n_2})(u_1,v_2)\cup (u_1,v_2)y$. If $d_{T_j}(u_1,u_2)=2$, set $A_i=x(\eta_{j},v_i)\cup (\eta_{j},v_i)y$ for $1\leq i\leq n_2$. It remains to consider the case that $d_{T_j}(u_1,u_2)\geq3$. We first find out a $V(H(\eta_{j}))V(H(\beta_{j}))$-linkage $D_1,D_2,\cdots, D_{n_2}$ by the three-type edges according to $\eta_{j}T_j\beta_{j}$. Thus $A_i=x(\eta_{j},v_i)\cup D_i\cup (\beta_{j},v_i)y$, where $1\leq i\leq n_2$ and the subscript $i$ of $v_i$ is expressed module $n_2$ as one of $1,2,\cdots,n_2$. It remains to consider the case that $T_j$ is of type $\uppercase\expandafter{\romannumeral2}$, denote the neighbor of $u_1$, $u_2$, $u_3$ in $T_j$ by $\eta_{j}$, $\beta_{j}$, $\gamma_{j}$ and the only one three-degree vertex in $T_j$ by $w_j$($\eta_{j}$, $\beta_{j}$, $\gamma_{j}$ and $w_j$ may be the same vertex). We find a $V(H(\eta_{j}))V(H(\beta_{j}))$-linkage and a $V(H(\gamma_{j}))V(H(w_j))$-linkage respectively by three-type edges of $G\circ H$. And join $x$, $y$, $z$ respectively to $H(\eta_{j})$, $H(\beta_{j})$ and $H(\gamma_{j})$. Thus, we are able to construct $n_2$ edge-disjoint $S$-trees corresponding to $T_j$. Since $1\leq j\leq \ell_1$, thus $\ell_1{n_2}$ edge-disjoint $S$-trees are constructed on Stage $\uppercase\expandafter{\romannumeral2}$.

{\bf Case 2.}~~ Only two of $x$, $y'$, $z'$ are the same vertex in $H(u_1)$.

We only consider the case of $x=y'$ (The other cases when $x=z'$ or $y'=z'$ can be proved with similar arguments). We may assume $x=(u_1,v_1)$, $y=(u_2,v_1)$, $z=(u_3,v_2)$. Since $\lambda_3(H)=\ell_2$, there exist $\ell_2$ edge-disjoint $v_1v_2$-paths $P_1,P_2,\cdots, P_{\ell_2}$ in $H$ such that $\ell(P_1)\leq \ell(P_2)\leq \cdots \leq \ell(P_{\ell_2})$. For $1\leq i\leq \ell_2$, denote the vertex in $P_i$ adjacent to $v_1$ by $\alpha_i$, and the vertex in $P_i$ adjacent to $v_2$ by $\beta_i$, and denote by $P_i(u_3)$ in $H(u_3)$  corresponding to $P_i$. Since $\lambda_3(G)=\ell_1$, there are $\ell_1$ edge-disjoint $\{u_1,u_2,u_3\}$-trees $T_1,T_2,\cdots, T_{\ell_1}$ in $G$.

On Stage $\uppercase\expandafter{\romannumeral1}$, fix $T_1$. If $\ell(P_i)\geq 2$ for each $i$ with $1\leq i\leq \ell_2$, let $T_i^{*}=x(u_1,\alpha_{i})\cup y(u_2,\alpha_{i})\cup zP_i(u_3)(u_3,\alpha_{i}) \cup T_1(\alpha_{i})$. Otherwise, $\ell(P_1)=1$, that is, $v_1$ is adjacent to $v_2$, then $d_{H}(v_1)\geq \ell_2+1$ or $d_{H}(v_2)\geq \ell_2+1$. If $d_{H}(v_1)\geq \ell_2+1$, denote a neighbor of $v_1$ by $\alpha_{\ell_2+1}$ which is not $\alpha_i$ ($1\leq i\leq \ell_2$). Then $T_1^{*}=\{x(u_1,\alpha_{\ell_2+1}), y(u_2,\alpha_{\ell_2+1}), zx''', x'''(u_3,\alpha_{\ell_2+1})\}\cup T_1(\alpha_{\ell_2+1})$. If $d_{H}(v_2)\geq \ell_2+1$, denote the neighbor of $v_2$ by $\beta_{\ell_2+1}$ which is not $\beta_i$ ($1\leq i\leq \ell_2$). Then $T_1^{*}=\{xz',z'(u_1,\beta_{\ell_2+1}), yz'',z''(u_2,\beta_{\ell_2+1}), z(u_3,\beta_{\ell_2+1})\}\cup T_1(\beta_{\ell_2+1})$.

On Stage $\uppercase\expandafter{\romannumeral2}$, $\ell_1n_2$ edge-disjoint $S$-trees are constructed with similar arguments as Case 1.

{\bf Case 3.}~~ $x$, $y'$, $z'$ are distinct vertices in $H(u_1)$.

Assume that $x=(u_1,v_1)$, $y=(u_2,v_2)$, $z=(u_3,v_3)$. Let $S'=\{v_1,v_2,v_3\}$ and $S''=\{u_1,u_2,u_3\}$.

Since $\lambda_3(H)=\ell_2$, there are $\ell_2$ edge-disjoint $S'$-trees $T_1,T_2,\cdots, T_{\ell_2}$ in $H$. For $1\leq i\leq \ell_2$, denote by $\alpha_i$ the vertex in $T_i$ adjacent to a vertex $v_1$ in $S'$, and $\ell(T_i)$ denotes the number of edges in $T_i$, and denote by $T_i(u_2)$ ($T_i(u_3)$) in $H(u_2)$ ($H(u_3)$) the tree corresponding to $T_i$. Since $\lambda_3(G)=\ell_1$, there are $\ell_1$ edge-disjoint $S''$-trees $T_1',T_2',\cdots, T_{\ell_1}'$ in $G$.

On Stage $\uppercase\expandafter{\romannumeral1}$, fix $T_1'$. If $\ell (T_i)\geq 3$ for each $i$ with $1\leq i\leq \ell_2$, let $T_i^{*}=x(u_1,\alpha_{i})\cup yT_i(u_2)(u_2,\alpha_{i})\cup zT_i(u_3)(u_3,\alpha_{i})\cup T'_1(\alpha_{i})$. Otherwise, similar to Case 2 of Lemma \ref{lem2}, the difficult case is that there is an $S'$-tree of length 2. Suppose $\ell (T_1)=2$ and assume $d_{T_1}(v_2)=2$. Thus $T_1^{*}$ has three structures as shown in Figure 7 where $T'_1$ is of type $\uppercase\expandafter{\romannumeral2}$ in Figure 8($a$); $T'_1$ is of type $\uppercase\expandafter{\romannumeral1}$ and $d_{T'_1}(v_2)=1$ in Figure 8($b$); $T'_1$ is of type $\uppercase\expandafter{\romannumeral1}$ and $d_{T'_1}(v_2)=2$ in Figure 8($c$).

\begin{figure}[h,t,b,p]
\begin{center}
\scalebox{0.9}[0.9]{\includegraphics{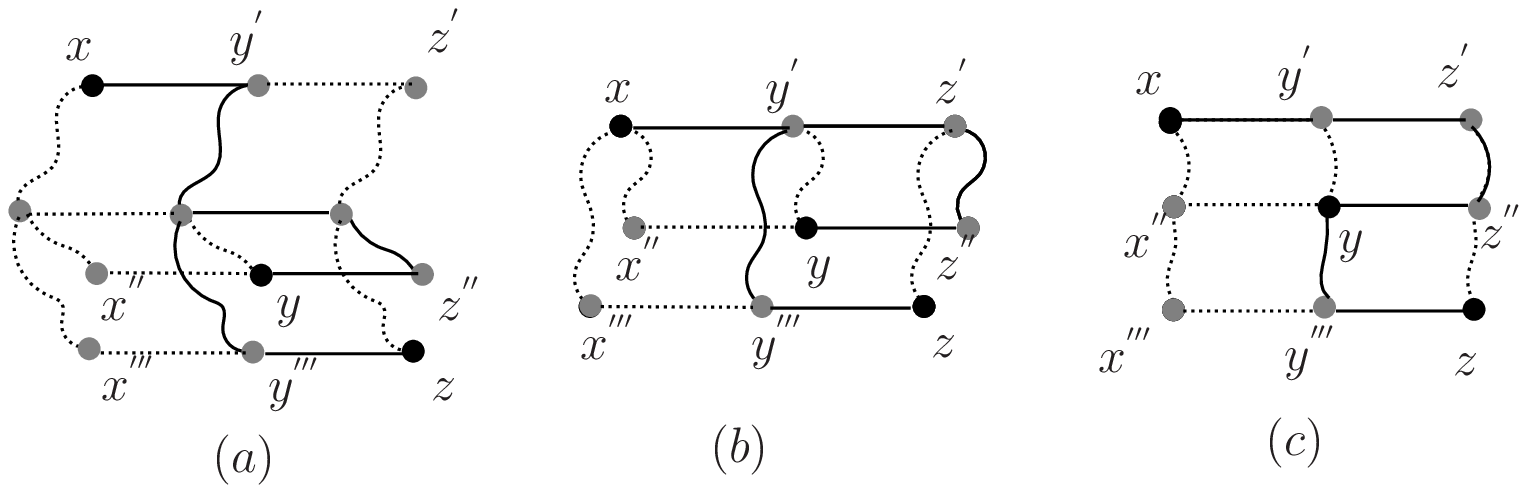}}\\[20pt]

Figure~8. One $S$-tree corresponding to $T'_1$,\\~~~~ the solid lines stand for the edges of the $S$-tree.
\end{center}
\end{figure}

On Stage $\uppercase\expandafter{\romannumeral2}$, $\ell_1n_2$ edge-disjoint $S$-trees are constructed with similar arguments as Case 1.

By Observation \ref{obs2}, in each case, these $\ell_2+\ell_1n_2$ $S$-trees are edge-disjoint, as desired.
\end{proof}

By combining the preceding three lemmas, we get the following result.

\begin{thm}\label{thm1}
Let $G$ and $H$ be two non-trivial graphs and $G$ is connected. Then $\lambda_3(G\circ H)\geq \lambda_3(H)+\lambda_3(G)|V(H)|$. Moreover, the lower bound is sharp.
\end{thm}

We know that the lower bounds of Theorem \ref{thm1} is sharp by the following corollary.

\begin{cor}\label{cor1}
$\lambda_3(P_{n_1}\circ P_{n_2})=n_2+1$.
\end{cor}

\begin{proof}
By Theorem \ref{thm1}, $\lambda_3(P_{n_1}\circ P_{n_2})\geq n_2+1$. On the other hand, by Observation \ref{lambdadelta}, $\lambda_3(P_{n_1}\circ P_{n_2})\leq \delta(P_{n_1}\circ P_{n_2})=n_2+1$. Thus $\lambda_3(P_{n_1}\circ P_{n_2})=n_2+1$.
\end{proof}

\section{Upper bound of $\lambda_3(G\circ H)$}
In this section, we give the upper bound of generalized 3-edge-connectivity of the lexicographic product of two graphs.

Yang and Xu \cite{Yang} investigated the classical edge-connectivity of
the lexicographic product of two graphs.

\begin{thm}\label{th4}{\upshape\cite{Yang}}
Let $G$ and $H$ be two non-trivial graphs and $G$ is connected, then $$\lambda(G\circ H)=\min\{\lambda(G)|V(H)|^2,\delta(H)+\delta(G)|V(H)|\}.$$
\end{thm}

In \cite{LLMS}, the sharp lower bound of the generalized 3-edge-connectivity of a graph is given as follows.

\begin{pro}\cite{LLMS}\label{lowerlambda}
Let $G$ be a connected graph with $n$ vertices. For every two
integers $s$ and $r$ with $s\geq 0$ and $r\in \{0,1,2,3\}$, if
$\lambda(G)=4s+r$, then $\lambda_3(G)\geq
3s+\lceil\frac{r}{2}\rceil$. Moreover, the lower bound is sharp. We
simply write $\lambda_3(G)\geq \frac{3\lambda-2}{4}$.
\end{pro}

From the above two results, we get the upper bounds of $\lambda_3(G\circ H)$.

\begin{thm}\label{thm5}
Let $G$ and $H$ be two non-trivial graphs and $G$ is connected. Then
$$
\lambda_3(G\circ H)\leq \min \Big \{\Big \lfloor
\frac{4\lambda_3(G)+2}{3}\Big \rfloor
|V(H)|^2,\delta(H)+\delta(G)|V(H)|\Big \}.
$$
Moreover, the upper bound is sharp.
\end{thm}

\begin{proof}
By Proposition \ref{lowerlambda}, $\lambda(G)\leq \lfloor\frac{4\lambda_3(G)+2}{3}\rfloor$.
By Proposition \ref{lambdalambda} and Theorem \ref{th4}, we have $\lambda_3(G\circ H)\leq\lambda(G\circ H)=\min\{\lambda(G)|V(H)|^2,\delta(H)+\delta(G)|V(H)|\}$. It follows that $\lambda_3(G\circ H)\leq \min \Big \{\Big \lfloor\frac{4\lambda_3(G)+2}{3}\Big \rfloor
|V(H)|^2,\delta(H)+\delta(G)|V(H)|\Big \}$. Moreover, the example in Corollary \ref{cor1} shows that the upper bound is sharp. The proof is now complete.
\end{proof}


\begin{thebibliography}{11}

\bibitem{Barden} B. Barden, R. Libeskind-Hadas, J. Davis, W. Williams, \emph{On edge-disjoint spanning trees in hypercubes},
Information Processing Letters 70(1999), 13-16.

\bibitem{Beineke1} L. W. Beineke, R.J. Wilson, \emph{Topics in Structural Graph Theory}, Cambrige University Press, 2013.

\bibitem{Bondy} J.A. Bondy, U.S.R. Murty, \emph{Graph Theory}, GTM 244, Springer, 2008.

\bibitem{Bresar}  B. Bre\v{s}ar, S. \v{S}pacapan, \emph{Edge connectivity of strong products of graphs}, Discuss. Math. Graph Theory 27(2007), 333-343.

\bibitem{Feng} M. Feng, M. Xu, K. Wang, \emph{Identifying codes of lexicographic product of graphs}, Electron. J. Combin. 19(4)(2012), 56-63.

\bibitem{Fragopoulou} P. Fragopoulou, S.G. Akl, \emph{Edge-disjoint spanning trees on the star network with applications to fault tolerance},
IEEE Trans. Computers, 45(2)(1996), 174-185.

\bibitem{Grotschel1} M. Gr\"{o}tschel, \emph{The Steiner tree packing problem in $VLSI$ design}, Math. Program. 78(1997), 265-281.

\bibitem{Hammack} R. Hammack, W. Imrich, S. Klav\u{z}ar, \emph{Handbook of Product Graphs}, 2nd Edition, CRC Press, 2011.

\bibitem{WS} W. Imrich, S. Klav\v{z}ar, \emph{Product Graphs: Structure and Recognition}, Wiley, New York, 2000.

\bibitem{Itai} A. Itai, M. Rodeh, \emph{The multi-tree approach to reliability in distributed networks}, Information and Computation,
79(1988), 43-59.

\bibitem{Klavzar}  S. Klav\v{z}ar, S. \v{S}pacapan, \emph{On the edge-connectivity of Cartesian product graphs}, Asian-Eur. J. Math. 1(2008), 93-98.

\bibitem{Kriesell1} M. Kriesell, \emph{Edge-disjoint trees containing some given vertices in a graph}, J. Combin. Theory, Ser.B, 88(2003), 53-65.

\bibitem{Ku} S. Ku, B. Wang, T. Hung, \emph{Constructing edge-disjoint spanning trees in product networks}, Parallel and Distributed Systems,
IEEE Transactions on Parallel and Disjoited Systems, 14(3)(2003),
213-221.

\bibitem{LXZW} F. Li, Z. Xu, H. Zhao, W. Wang, \emph{On the number of spanning trees of the lexicographic product of networks}, Sci. China, Ser.F, 42(2012), 949-959.

\bibitem{LLSun} H. Li, X. Li, Y. Sun, \emph{The generalied $3$-connectivity of Cartesian product graphs}, Discrete Math. Theor. Comput. Sci.
14(1)(2012), 43-54.

\bibitem{LM4} X. Li, Y. Mao, \emph{The minimal size of a graph with given generalized 3-edge-connectivity}, arXiv: 1201.3699[math.CO] 2012.

\bibitem{LLMS} X. Li, Y. Mao, Y. Sun, \emph{On the generalized (edge-)connectivity of graphs}, Australasian J. Combin. 58(2014), 304-319.

\bibitem{Nash} C.St.J.A. Nash-Williams, \emph{Edge-disjonint spanning trees of finite graphs}, J. London Math. Soc. 36(1961), 445-450.

\bibitem{Oellermann1} O.R. Oellermann, \emph{Connectivity and edge-connectivity in graphs: A survey}, Congessus Numerantium 116(1996), 231-252.

\bibitem{OY} K. Ozeki, T. Yamashita, \emph{Spanning trees: A survey}, Graphs Combin. 27(1)(2011), 1-26.

\bibitem{Palmer} E. Palmer, \emph{On the spanning tree packing number of a graph: A survey}, Discrete Math. 230(2001), 13-21.

\bibitem{Roskind} J. Roskind, R. Tarjan, \emph{A note on finding maximum-cost edge-disjoint spanning trees}, Math. Operations
Research, 10(2)(1985), 701-708.

\bibitem{Tutte} W. Tutte, \emph{On the problem of decomposing a graph into $n$ connected factors}, J. London Math. Soc. 36(1961), 221-230.

\bibitem{Volkmann} L. Volkmann, \emph{Edge connectivity in $p$-partite graphs}, J. Graph Theory 13(1)(1989) 1-6.

\bibitem{Wang} H. Wang, D. Blough, \emph{Construction of edge-disjoint spanning trees in the torus and application to multicast in wormhole-routed
networks}, Proc. Int'l Conf. Parallel and Distributed Computing
Systems, 1999.

\bibitem{West} D. West, H. Wu, \emph{Packing Steiner trees and $S$-connectors in graphs}, J. Combin. Theory, Ser.B, 102(2012), 186-205.

\bibitem{XY} J. Xu, C. Yang, \emph{Connectivity of Cartesian product graphs}, Discrete Math. 306(2006), 159¨C165.

\bibitem{Yang} C. Yang, J. Xu, \emph{Connectivity of lexicographic product and direct product of graphs}, Ars Combin. 111(2013), 3-12.

\end{thebibliography}
\end{document}